\documentclass[11pt]{article}
\usepackage{amscd}
\usepackage{amsfonts}
\usepackage{amsmath}
\usepackage{amssymb}
\usepackage{amsthm}
\usepackage{bbm}
\usepackage{CJK}
\usepackage{fancyhdr}
\usepackage{graphicx}
\usepackage{hyperref}
\usepackage{indentfirst}
\usepackage{latexsym}
\usepackage{mathrsfs}
\usepackage{xypic}
\usepackage{tikz-cd}

\newtheorem{theorem}{Theorem}[section]
\newtheorem{lemma}[theorem]{Lemma}
\newtheorem{definition}[theorem]{Definition}
\newtheorem{proposition}[theorem]{Proposition}
\newtheorem{example}[theorem]{Example}

\usepackage[top=1in,bottom=1in,left=1.25in,right=1.25in]{geometry}
\def\<{\langle}
\def\>{\rangle}

\date{}
\begin{document}
\renewcommand{\baselinestretch}{1.2}
\renewcommand{\arraystretch}{1.0}
\title{\bf Nijenhuis operators on Leibniz algebras }\author{{\bf Bibhash Mondal$^{1}$,     Ripan Saha$^{2}$\footnote
      { Corresponding author (Ripan Saha),  Email: ripanjumaths@gmail.com}}\\
 {\small 1. Department of Mathematics, Behala College}\\
  {\small Behala, 700060, Kolkata, India}\\  
 {\small Email: mondaliiser@gmail.com}\\
 {\small 2. Department of Mathematics, Raiganj University} \\
{\small  Raiganj 733134, West Bengal, India}}
 \maketitle
\begin{center}
\begin{minipage}{13.cm}

{\bf \begin{center} ABSTRACT \end{center}}
In this paper, we study Nijenhuis operators on Leibniz algebras. We discuss the relationship of Nijenhuis operators with Rota-Baxter operators and modified Rota-Baxter operators on Leibniz algebras. We define a representation theory of Nijenhuis Leibniz algebras and construct a cohomology theory. Next, we define a one-parameter formal deformation theory of Nijenhuis Leibniz algebras and study infinitesimals, rigidity, and equivalences along the line of Gerstenhaber deformation theory. As an application of our cohomology theory, we show that our cohomology is deformation cohomology and study abelian extensions of such algebras.

 \medskip

 \smallskip

 \end{minipage}
 \end{center}
 \normalsize\vskip0.5cm

\section{Introduction}
\def\theequation{\arabic{section}. \arabic{equation}}
\setcounter{equation} {0}
J. L. Loday \cite{Loday} introduced Leibniz algebra in 1993. Leibniz algebra is a $\mathbb{K}$-vector space $\mathfrak{g}$ together with a bilinear operation (called the bracket) $ [ ~,~]: \mathfrak{g} \times \mathfrak{g} \rightarrow \mathfrak{g}$ satisfying the following identity:  \[ [x,[y,z]]=[[x,y],z]+[y,[x,z]] ,~~~~ \mbox{for}~ x,y,z \in \mathfrak{g}. \]

Leibniz algebra is often considered as noncommutative Lie algebra; as a result several research have been done to generalize the result of Lie algebra to Leibniz algebra. In recent years, various linear operators like Rota-Baxter operator \cite{MS}, modified Rota Baxter operator \cite{MS-mody} etc., and their induced cohomology, deformation is studied.
In this paper, we study Nijenhuis operators on Leibniz algebras. 

A Nijenhuis operator on $\mathfrak{g}$ is a linear map $ N:\mathfrak{g} \rightarrow \mathfrak{g}$ satisfying the following condition \[[N(x),N(y)]=N \bigg([N(x),y]+[x,N(y)]- N[x,y]\bigg),~~~~\mbox{for all} ~~x, y \in \mathfrak{g}.\]

Dorfman \cite{Dorfman1993} studied Nijenhuis operator via deformation of Lie algebra. Moreover, Nijenhuis operators on Lie algebras play an interesting role in the study of integrability of nonlinear evolution equations \cite{Dorfman1993}. Introduction of Dirac structures by Dorfan gave new interpretations to the already existing Nijenhuis set ups. In 2004, Gallardo and Nunes \cite{CN04} introduced Dirac Nijenhuis structures. Longguang and Baokang \cite{LB04} separately developed Dirac Nijenhuis manifolds in 2004. In 2011, Kosmann-Schwarzbach \cite{kos} studied Dirac Nijenhuis structures on Courant algebroid. A notion of a Nijenhuis operator on a 3-Lie algebra was introduced in \cite{zhang} to study the 1-order deformations of a 3-Lie algebra. In \cite{Wang}, the authors defined the notion of a Nijenhuis operator on a pre-Lie algebra which generates a trivial deformation of the pre-Lie algebra. Das and Sen \cite{DS} studied Nijenhuis operators on Hom-Lie algebras from cohomological point of view. For the study of Nijenhius operators of various other algebraic structures, see, \cite{LSZB, HCM}.

 In this paper, we study Nijenhuis operators on Leibniz algebras.  We call a Leibniz algebra with a Nijenhuis operator defined on it by Nijenhuis Leibniz algebra. We give an interesting example of such algebras in dimension 2. We discuss relationship between Nijenhuis operator with Rota-Baxter operator and modified Rota-Baxter operator defined on a Leibniz algebra. It is believed that one can learn more about a mathematical object by studying its deformations \cite{gers, gers2}. To study deformation theory of a type of algebra one needs a suitable cohomology, called deformation cohomology which controls deformations in question. In the case of associative algebras, deformation cohomology is Hochschild cohomology and for Lie algebras, the associated deformation cohomology is Chevalley-Eilenberg cohomology. In this paper, we study one-parameter formal deofrmation theory of Nijenhuis Leibniz algebras and associated deformation cohomology. We study the infinitesimal, rigidity and equivalences for such deformation.  We also show that our cohomology is related with the abelian extensions of such algebras.

The work of the present paper is organised as follows: In Chapter $\ref{chap2}$, we recall some definitions, give an example, and show some relationship with other operators. In Chapter \ref{chap3}, we define a representation of Nijenhuis Leibniz algebras to define a suitable cohomology theory. In Chapter \ref{chap4}, we define a cohomology theory for Nijenhuis Leibniz algebras generalizing the cohomology of Leibniz algebras. In Chapter \ref{chap5}, we define formal deformation theory of such algebras following the Gerstenhaber \cite{gers, gers2} algebraic deformation theory. In the final Chapter \ref{chap6}, we study an abelian extension of Nijenhuis Leibniz algebras.

Throghout the paper, all vector spaces are considered over field $\mathbb{K}$ of characteristic 0.

 \medskip

\section{Basics and examples of Nijenhuis operators}\label{chap2}
\def\theequation{\arabic{section}.\arabic{equation}}
\setcounter{equation} {0}
Let us recall some basic definitions, a new example and few small results related to Nijenuis Leibniz algebras \cite{Loday, Wang, MS, MS-mody}.
\begin{definition}(\cite {Loday})
A \textbf{ Leibniz algebra} is a vector space $\mathfrak{g}$ together with a bilinear operation (called the bracket) $ [ ~,~]: \mathfrak{g} \times \mathfrak{g} \rightarrow \mathfrak{g}$ satisfying the following identity  \[ [x,[y,z]]=[[x,y],z]+[y,[x,z]] ,~~~~ \mbox{for}~ x,y,z \in \mathfrak{g}. \]
It is denoted by $(\mathfrak{g},[~,~])$ or sometimes simply by $\mathfrak{g}$. The above definition is of left Leibniz algebra. In this paper, we consider left Leibniz algebra simply as Leibniz algebra.
\end{definition}

\begin{example} (\cite{Demir})\label{exam leib}
Consider the two dimensional vector space $\mathbb{R}^2$ with standard basis $\{e_1,e_2\}$. Now the bracket is defined by $[e_2,e_1]=[e_2,e_2]=e_1$ with all other combination is zero. Then $(\mathbb{R}^2, [~,~])$ is a Leibniz algebra.
\end{example}

\begin{definition}(\cite {Sun})
Let $\mathfrak{g}$ be a Leibniz algebra. A \textbf{Nijenhuis operator } on $\mathfrak{g}$ is a linear map $ N:\mathfrak{g} \rightarrow \mathfrak{g}$ satisfying the following condition \[[N(x),N(y)]=N \bigg([N(x),y]+[x,N(y)]- N[x,y]\bigg),~~~~\mbox{for all} ~~x, y \in \mathfrak{g}\]
\end{definition}
\begin{definition}\cite{Tang}
Let $\mathfrak{g}$ be a Leibniz algebra. A linear operator $N$ on $\mathfrak{g}$ is called  a \textbf{Rota-Baxter operator} on $(\mathfrak{g}$ if $[Nx,Ny]=N([x,Ny]+[Nx,y])$ for all $x,y \in \mathfrak{g}.$
\end{definition}
\begin{definition} (\cite {Guo})
Let $\mathfrak{g}$ be a Leibniz algebra. A linear operator $N$ on $\mathfrak{g}$ is called  a \textbf{Rota-Baxter operator of weight $\lambda \in \mathbb{K}$} on $\mathfrak{g}$ if $[Nx,Ny]=N([x,Ny]+[Nx,y] + \lambda N[x,y])$, for all $x,y \in \mathfrak{g}.$
\end{definition}
\begin{definition}(\cite{Das})
Let $\mathfrak{g}$ be a Leibniz algebra. A \textbf{modified Rota-Baxter operator of weight $\lambda$} on $\mathfrak{g}$ is a linear operator $N: \mathfrak{g}\rightarrow \mathfrak{g}$, such that $[Nx,Ny]=N([x,Ny]+[Nx,y])+\lambda [x,y]$ for all $x,y \in \mathfrak{g}.$
\end{definition}

\begin{definition}
 A Leibniz algebra  $\mathfrak{g}$ equipped with a Nijenhuis operator $N$ on $\mathfrak{g}$ is called a \textbf{Nijenhuis Leibniz algebra} and it is denoted by $(\mathfrak{g}_N,[~,~])$ or simply by $\mathfrak{g}_N$
\end{definition}
\begin{definition}
Let $\mathfrak{g}_N$ and $(\mathfrak{g^{'}}_{N^{'}},[~,~]^{'})$ be two Nijenhuis Leibniz algebra. Then a map $\phi : \mathfrak{g} \rightarrow \mathfrak{g^{'}}$ is called a \textbf{morphism of Nijenhuis Leibniz algebra} if the map $\phi$ is a Leibniz algebra morphism satisfying the condition $N^{'} \circ \phi = \phi \circ N.$
\end{definition}

\begin{proposition}\label{prop basic}
Let $\mathfrak{g}_N$ be a Nijehnuis Leibniz algebra. We define $[x,y]_*=[x,N(y)]+[N(x),y] -N([x,y])$ for all $x ,y  \in \mathfrak{g}$. Then, 
\begin{enumerate}
\item  $(g,[~,~]_*)$ is a Leibniz algebra.
\item $N$ is also a Nijehnuis operator on $(\mathfrak{g},[~,~]_*)$
\item The map $N:(g,[~,~]) \rightarrow (g,[~,~]_*)$ is a morphism of Nijenhuis Leibniz algebra.
\end{enumerate}
\end{proposition}
 
\begin{proof}
\begin{enumerate}

\item
It is obvious to observe that $[~,~]_*$ is a bilinear map. Now we will show that the Leibniz identity hold for this bilinear map that is the identity \[ [a,[b,c]_{*}]_{*}=[[a,b]_{*},c]_{*}+[b,[a,c]_{*}]_{*} ,~~~~ \mbox{for}~ a,b,c \in \mathfrak{g} \] holds.
Now, 
\begin{eqnarray*}
[a,[b,c]_{*}]_{*}\\ 
&=& [Na,[b,c]_{*}]+[a,N([b,c]_{*})] -N([a,[b,c]_*])\\
&=& [Na,[Nb,c]]+[Na,[b,Nc]]-[Na,N([b,c])]+[a,[Nb,Nc]]\\
&&-N([a,[Nb,c]])-N([a,[b,Nc]])+N([a,N([b,c])])\\
&=&[Na,[Nb,c]]+[Na,[b,Nc]]+ N([a,[b,c]])\\
&&-N([Na,[b,c]])-N([a,N[b,c]])+[a,[Nb,Nc]]\\
&&-N([a,[Nb,c]])-N([a,[b,Nc]])+N([a,N([b,c])])\\
&=&[Na,[Nb,c]]+[Na,[b,Nc]]+[a,[Nb,Nc]]+N([a,[b,c]])\\
&&-N([Na,[b,c]])-N([a,[Nb,c]])-N([a,[b,Nc]])\\
&=&[Na,[Nb,c]]+[Na,[b,Nc]]+[a,[Nb,Nc]]+N([a,[b,c]])\\
&&-N([Na,[b,c]])+[a,[Nb,c]]+[a,[b,Nc]]).\\
\end{eqnarray*}
In the same manner as above, we have 
\begin{align*}
&[[a,b]_{*},c]_{*}= [[Na,Nb],c]+[[Na,b],Nc]+[[a,Nb],Nc]+N([[a,b],c])\\
&-N([[Na,b],c]+[[a,Nb],c]+[[a,b],Nc]).
\end{align*}
and
\begin{align*}
[b,[a,c]_{*}]_{*}&= [Nb,[Na,c]]+[Nb,[a,Nc]]+[b,[Na,Nc]]+N([b,[a,c]])\\
&-N([Nb,[a,c]])+[b,[Na,c]]+[b,[a,Nc]]).
\end{align*}
Now using the Leibniz identity of the Leibniz algebra $(\mathfrak{g},[~,~]_{\mathfrak{g}})$, we conclude 
\[ [a,[b,c]_{*}]_{*}=[[a,b]_{*},c]_{*}+[b,[a,c]_{*}]_{*} ,~~~~ \mbox{for}~ a,b,c \in \mathfrak{g} \] holds. Hence, $(\mathfrak{g},[~,~]_*)$ is a Leibniz algebra.
 
\item  Again, for any $a,b \in \mathfrak{g}$, we have 
\begin{align*}
[Na,Nb]_{*}\\
&=[N(Na),Nb]+[Na,N(Nb)]-N([Na,Nb])\\
&=N([N(Na),b]+[Na,Nb]-N([Na,b]))\\
&~~+N([Na,Nb]+[a,N(Nb)]-N([a,Nb]))\\
&~~-N^2([Na,b]+[a,Nb]-N([a,b]))\\
=&N\bigg(\ [N(Na),b]+[Na,Nb]-N([Na,b])\\
&~~+[Na,Nb]+[a,N(Nb)]-N([a,Nb])\\
&~~-N([Na,b]+[a,Nb]-N([a,b]))\bigg)\\
&=N\bigg( [Na,b]_*+[a,Nb]_*-N([a,b]_*)\bigg ).
\end{align*}
Hence, $N$ is also a Nijenhuis operator on the Leibniz algebra $(\mathfrak{g},[~,~]_*).$
\item  
It is obvious to observe that the map $N: (\mathfrak{g}_{N},[~,~]) \rightarrow (\mathfrak{g}_{N},[~,~]_*)$ , $g \mapsto N(g)$ is a morphism of Nijenhuis Leibniz algebra.
\end{enumerate} 
\end{proof}
\begin{example}
Let us consider the Leibniz algebra $(\mathbb{R}^2,[~,~])$ defined in \ref{exam leib}. Now consider the linear map $N : \mathbb{R}^2 \rightarrow \mathbb{R}^2$ defined by $N(x)=Ax$, where $A=\begin{bmatrix}
a & b\\
c & d \\
\end{bmatrix}$. Then, $N$ is Nijenhuis operator if and only if $c=0$ and either $a=d$ or $a-d=b$.\\
To prove the above , Let $N$ be a Nijenhuis operator. Thus,   
\[[N(x),N(y)]=N([N(x),y]+[x,N(y)]- N[x,y]),~~~~\mbox{for all} ~~x, y \in \mathfrak{g}.\]
Observe that, we have
\begin{itemize}
\item[(i)]
If $x=e_1$ and $y=e_1$, we have
\begin{align*}
&[ae_1+ce_2,ae_1+ce_2]=N([ae_1+ce_2,e_2]),\\
&(ac+c^2)e_1=cNe_1=c(ae_1+ce_2),\\
&\text{Therefore,}~ c=0
\end{align*}
\item[(ii)] 
If $x=e_1$ and $y=e_2$, we have
\begin{align*}
&[ae_1+ce_2,be_1+de_2]=N([ae_1+ce_2,e_2]),\\
&(bc+cd)e_1=cNe_1=c(ae_1+ce_2).
\end{align*}
Since, $c=0$, we have no condition from here.
\item[(iii)] 
If $x=e_2$ and $y=e_1$, we have
\begin{align*}
&[be_1+de_2,ae_1]=N([be_1+de_2,e_1]+[e_2,ae_1]-N[e_2,e_1]),\\
&ade_1=N(de_1+ae_1-ae_1)=ade_1.
\end{align*}
Hence, this also gives no additional condition.
\item[(iv)] 
If $x=e_2$ and $y=e_2$, we have
\begin{align*}
&[be_1+de_2,be_1+de_2]=N([be_1+de_2,e_2]+[e_2,be_1+de_2]-N[e_2,e_2]),\\
& (bd+d^2)e_1=N(de_1+be_1+de_1-ae_1),\\
& bd+d^2=(2d+b-a)a ~~\mbox{ i.e.}~ (a-d)^2=b(a-d). 
\end{align*}
This implies either $a=d$ or $a=b+d.$ The converse is a routine verification.
\end{itemize}
\end{example}
Now, we have the following proposition (similar result for pre lie algebra in Prop. 5.2 of \cite{Wang}).
\begin{proposition}
Let $N: \mathfrak{g} \rightarrow \mathfrak{g}$ is a linear operator where $\mathfrak{g}$ is a Leibniz algebra. Then
\begin{itemize}
\item[(a)] If $N^2=0$ then $N$ is a nijenhuis operator if and only if $N$ is a Rota-Baxter operator.
\item[(b)] If $N^2=N$ then $N$ is a nijenhuis operator if and only if $N$ is a Rota-Baxter operator of weight $-1.$
\item[(c)] If $N^2=\pm Id$ then $N$ is a nijenhuis operator if and only if $N$ is a modified Rota-Baxter operator of weight $\mp 1.$
\end{itemize}
\end{proposition}
\begin{proof}
\begin{enumerate}
\item Let $N^2=0.$ Suppose $N$ is a Nijenhuis operator. Then for any $x,y \in \mathfrak{g}$, we have 
\begin{align*}
N(x),N(y)]&=N([N(x),y]+[x,N(y)]- N[x,y])\\
&=N([N(x),y]+[x,N(y)])-N^2[x,y]\\
&=N([N(x),y]+[x,N(y)])
\end{align*}
Hence, $N$ is a Rota-Baxter operator. Proof of the converse part is similar.
 \item Let $N^2=N.$ Suppose $N$ is a Nijenhuis operator. Then  for any $x,y \in \mathfrak{g}$ we have 
\begin{align*}
N(x),N(y)]&=N([N(x),y]+[x,N(y)]- N[x,y])\\
&=N([N(x),y]+[x,N(y)])-N^2[x,y]\\
&=N([N(x),y]+[x,N(y)]-[x,y]).
\end{align*}
Hence, $N$ is a Rota-Baxter operator of weight $-1$. Similarly, the converse can be shown.
\item
Let $N^2=I.$ Suppose $N$ is a Nijenhuis operator. Then, for any $x,y \in \mathfrak{g}$, we have 
\begin{align*}
N(x),N(y)]&=N([N(x),y]+[x,N(y)]- N[x,y])\\
&=N([N(x),y]+[x,N(y)])-N^2[x,y]\\
&=N([N(x),y]+[x,N(y)])-[x,y].
\end{align*}
Hence, $N$ is a modified Rota-Baxter operator of weight $-1$. In a similar way the other cases can be shown.
\end{enumerate}
\end{proof}
\section{Representations of Nijenhuis Leibniz algebras}\label{chap3}
\begin{definition}(\cite{Loday})
Let $\mathfrak{g}$ be a Leibniz algebra. A $\mathbf{representation}$ of $\mathfrak{g}$ is a triple $(V,l_V,r_V)$ where $V$ is vector space together with linear maps (called the left and right $\mathfrak{g}$-actions respectively ) $l_V : \mathfrak{g}\otimes V \rightarrow V ~~\mbox{and}~~ r_V : V\otimes \mathfrak{g} \rightarrow V $ satisfying the following conditions

\[l_V(x,l_V(y,u))=l_V([x,y],u)+l_V(y,l_V(x,u))\]
\[l_V(x,r_V(u,y))=r_V(l_V(x,u),y)+r_V(u,[x,y])\]
\[ r_V(u,[x,y])= r_V(r_V(u,x),y)+l_V(x,r_V(u,y))\]
for all $x,y \in \mathfrak{g}$ and $u\in V.$

\end{definition}
\begin{definition}
Let $\mathfrak{g}_N$ be a Nijenhuis Leibniz algebra. A representation of $\mathfrak{g}_N$ is a quadruple $(V,l_V,r_V,N_V) $, where $(V,l_V,r_V)$ is a representation of the Leibniz algebra $\mathfrak{g}$ and $N_V: V \rightarrow V$ is a linear map satisfying the following conditions
\[l_V(Nx,N_V(u))=N_V\bigg(l_V(Nx,u)+l_V(x,N_V(u))-(N_V \circ l_V)(x,u)\bigg)\]
\[r_V(N_V(u),Nx)=N_V\bigg(r_V(N_V(u),x)+r_V(u,Nx)-(N_V \circ r_V)(u,x)\bigg)\]
for all $x \in \mathfrak{g}$ and $u \in V.$
\end{definition}

\begin{proposition} \label{prop coh}
Let $\mathfrak{g}_N$ be a Nijenhuis Leibniz algebra with representation $(V,l_{V},r_{V},N_V)$. We define $l_V^{'} : \mathfrak{g} \otimes V \rightarrow V  ~, ~ r_{V}^{'} :  V  \otimes \mathfrak{g}  \rightarrow V $ respectively by 

\[l^{'}_V (x,u)=l_V(Nx,u)-N_V(l_V(x,u))+l_V(x,N_V(u))\] 
\[r^{'}_V(u,x)=r_V(u,Nx)-N_V(r_V(u,x))+r_V(N_V(u),x)\]

 for all $x \in {\mathfrak{g}} , u \in V $. Then $(V,l_{V}^{'},r_V^{'},N_V)$ will be a representation of the Nijenhuis Leibniz algebra $(\mathfrak{g}_N,[~,~]_*)$.
\end{proposition}
\begin{proof}
For any $x,y \in \mathfrak{g}$ and $u \in V$, we have
\begin{align*}
&l_V^{'}(x,l_V^{'}(y,u))-l_V^{'}([x,y]_*,u)-l_V^{'}(y,l_V^{'}(x,u))\\
&=l_V(Nx,l_V^{'}(y,u))-(N_V \circ l_V)(x,l_V^{'}(y,u))+l_V(x,(N_V \circ l_V^{'})(y,u))\\
&-l_V(N[x,y]_*,u)+(N_V \circ l_V)([x,y]_*,u)-l_V([x,y]_*,N_V(u))\\
&-l_V(Ny,l_V^{'}(x,u))+(N_V \circ l_V)(y,l_V^{'}(x,u))-l_V(y,(N_V \circ l_V^{'})(x,u))\\
&=l_V(Nx,l_V(Ny,u))-l_V(Nx,(N_V \circ l_V)(y,u))+l_V(Nx,l_V(y,N_V(u)))\\
&-(N_V \circ l_V)(x,l_V(Ny,u))+(N_V \circ l_V)(x,(N_V \circ l_V)(y,u))-(N_V \circ l_V)(x,l_V(y,N_V(u)))\\
&+l_V(x,(N_V \circ l_V)(Ny,u))-l_V(x,(N_V \circ (N_V \circ l_V))(y,u))+l_V(x,(N_V \circ l_V)(y,N_V(u)))\\
&-l_V([Nx,Ny],u)+(N_V \circ l_V)([Nx,y],u)+(N_V \circ l_V)([x,Ny],u)-(N_V \circ l_V)(N[x,y],u)\\
&-l_V([Nx,y],N_V(u))-l_V([x,Ny],N_V(u))+l_V(N[x,y],N_V(u))\\
&-l_V(Ny,l_V(Nx,u))+l_V(Ny,(N_V \circ l_V)(x,u))-l_V(Ny,l_V(x,N_V(u)))\\
&+(N_V \circ l_V)(y,l_V(Nx,u))-(N_V \circ l_V)(y,(N_V \circ l_V)(x,u))+(N_V \circ l_V)(y,l_V(x,N_V(u)))\\
&-l_V(y,(N_V \circ l_V)(Nx,u))+l_V(y,(N_V \circ(N_V \circ l_V))(x,u))-l_V(y,(N_V \circ l_V)(x,N_V(u)))\\
&=\bigg(l_V(Nx,l_V(Ny,u))-l_V([Nx,Ny],u)-l_V(Ny,l_V(Nx,u))\bigg)\\
&-l_V(Nx,(N_V \circ l_V)(y,u))+l_V(Nx,l_V(y,N_V(u)))\\
&-(N_V \circ l_V)(x,l_V(Ny,u))+(N_V \circ l_V)(x,(N_V \circ l_V)(y,u))-(N_V \circ l_V)(x,l_V(y,N_V(u)))\\
&+l_V(x,(N_V \circ l_V)(Ny,u))-l_V(x,(N_V \circ (N_V \circ l_V))(y,u))+l_V(x,(N_V \circ l_V)(y,N_V(u)))\\
&+(N_V \circ l_V)([Nx,y],u)+(N_V \circ l_V)([x,Ny],u)-(N_V \circ l_V)(N[x,y],u)\\
&-l_V([Nx,y],N_V(u))-l_V([x,Ny],N_V(u))+l_V(N[x,y],N_V(u))\\
&+l_V(Ny,(N_V \circ l_V)(x,u))-l_V(Ny,l_V(x,N_V(u)))\\
&+(N_V \circ l_V)(y,l_V(Nx,u))-(N_V \circ l_V)(y,(N_V \circ l_V)(x,u))+(N_V \circ l_V)(y,l_V(x,N_V(u)))\\
&-l_V(y,(N_V \circ l_V)(Nx,u))+l_V(y,(N_V \circ(N_V \circ l_V))(x,u))-l_V(y,(N_V \circ l_V)(x,N_V(u)))\\
&=-(N_V \circ l_V)(Nx, l_V(y,u))-(N_V \circ l_V)(x,(N_V \circ l_V)(y,u))+(N_V \circ ( N_V \circ l_V))(x,l_V(y,u))\\
&+l_V(Nx,l_V(y,N_V(u)))-(N_V \circ l_V)(x,l_V(Ny,u))+(N_V \circ l_V)(x,(N_V \circ l_V)(y,u))\\
&-(N_V \circ l_V)(x,l_V(y,N_V(u)))+l_V(x,(N_V \circ l_V)(Ny,u))-l_V(x,(N_V \circ (N_V \circ l_V))(y,u))\\
&l_V(x,(N_V \circ l_V)(y,N_V(u)))+(N_V \circ l_V)([Nx,y],u)+(N_V \circ l_V)([x,Ny],u)\\
&-(N_V \circ l_V)(N[x,y],u)-l_V([Nx,y],N_V(u))-l_V([x,Ny],N_V(u))\\
&(N_V \circ l_V)(N[x,y],u)+(N_V \circ l_V)([x,y],N_V(u))-(N_V \circ (N_V \circ l_V))([x,y],u)\\
&(N_V \circ l_V)(Ny,l_V(x,u))+(N_V \circ l_V)(y,(N_V \circ l_V)(x,u))-(N_V \circ (N_V \circ l_V))(y,l_V(x,u))\\
&-l_V(Ny,l_V(x,N_V(u)))+(N_V \circ l_V)(y,l_V(Nx,u))-(N_V \circ l_V)(y,(N_V \circ l_V)(x,u))\\
&+(N_V \circ l_V)(y,l_V(x,N_V(u)))-l_V(y,(N_V \circ l_V)(Nx,u))+l_V(y,(N_V \circ (N_V \circ l_V))(x,u))\\
&-l_V(y,(N_V \circ l_V)(x,N_V(u)))\\
\end{align*}
\begin{align*}
=&\bigg(-(N_V \circ l_V)(Nx, l_V(y,u))+(N_V \circ l_V)([Nx,y],u)+(N_V \circ l_V)(y,l_V(Nx,u)) \bigg)\\
&+\bigg ( -(N_V \circ l_V)(x,(N_V \circ l_V)(y,u))+(N_V \circ l_V)(x,(N_V \circ l_V)(y,u)) \bigg )\\
&+\bigg ((N_V \circ ( N_V \circ l_V))(x,l_V(y,u))- ((N_V \circ ( N_V \circ l_V))([x,y],u)-((N_V \circ ( N_V \circ l_V))(y,l_V(x,u))\bigg)\\
&+\bigg (-( N_V \circ l_V)(N[x,y],u)+( N_V \circ l_V)(N[x,y],u)  \bigg )\\
&+\bigg (( N_V \circ l_V)(y,( N_V \circ l_V)(x,u))-( N_V \circ l_V)(y,( N_V \circ l_V)(x,u)) \bigg )\\
&+\Bigg ( \bigg( l_V(Nx,l_V(y,N_V(u)))-l_V([Nx,y],N_V(u)) \bigg)+\\
& \bigg ( -l_V(y,(N_V \circ l_V)(Nx,u))+l_V(y,(N_V \circ (N_V \circ l_V))(x,u))-l_V(y,(N_V \circ l_V)(x,N_V(u)))\bigg )\Bigg)\\
&+\bigg( -(N_V \circ l_V)(x,l_V(y,N_V(u)))+(N_V \circ l_V)([x,y],N_V(u))+(N_V \circ l_V)(y,l_V(x,N_V(u))) \bigg)\\
&+\bigg(- (N_V \circ l_V)(x,l_V(Ny,u))+(N_V \circ l_V)([x,Ny],u)+(N_V \circ l_V)(Ny,l_V(x,u))\bigg)\\
&+\Bigg( \bigg ( - l_V ([x,Ny],N_V(u))-l_V(Ny,l_V(x,N_V(u)))\bigg )+\\
& \bigg ( l_V(x,(N_V \circ l_V)(Ny,u))-l_V(x,(N_V \circ (N_V \circ l_V))(y,u))+l_V(x,(N_V \circ l_V)(y,N_V(u))) \bigg ) \Bigg)\\
&=0.\\
\end{align*}

Therefore,
\[l_V^{'}(x,l_V^{'}(y,u))=l_V^{'}([x,y]_{*},u)+l_V^{'}(y,l_V^{'}(x,u)).\]

Similarly, it can be shown that the following equations holds for all $x,y \in \mathfrak{g}, u \in V.$ 
\[l_V^{'}(x,r_V^{'}(u,y))=r_V^{'}(l_V^{'}(x,u),y)+r_V^{'}(u,[x,y]_{*}),\]
\[ r_V^{'}(u,[x,y]_{*})= r_V^{'}(r_V^{'}(u,x),y)+l_V^{'}(x,r_V^{'}(u,y)).\]

Thus, $(V,l_V^{'},r_V^{'} )$ is a representation of the Leibniz algebra $(\mathfrak{g}_N,[~,~]_*)$.
Observe that, we also have

\begin{align*}
&l_V^{'}(Nx,N_V(u))\\
&=l_V(N(Nx),N_V(u))-(N_V \circ l_V)(Nx,N_V(u))+l_V(Nx,N_V(N_V(u)))\\
&=(N_V \circ l_V)(N(N(x)),u)+(N_V \circ l_V)(Nx,N_V(u))-(N_V \circ (N_V \circ l_V))(Nx,u)\\
&-N_V \bigg ( (N_V \circ l_V)(Nx,u)+(N_V \circ l_V)(x, N_V(u))- (N_V \circ (N_V \circ l_V))(x,u) \bigg )\\
&+(N_V \circ l_V)(Nx,N_V(u))+(N_V \circ l_V)(x,N_V(N_V(u)))-(N_V \circ (N_V \circ l_V))(x,N_V(u))\\
&=N_V \bigg( l_V(N(Nx),u)+l_V(Nx,N_V(u))-(N_V \circ l_V)(Nx,u)\\
&-(N_V \circ l_V)(Nx,u)-(N_V \circ l_V)(x,N_V(u))+(N_V \circ (N_V \circ l_V))(x,u)\\
&+l_V(Nx,N_V(u))+l_V(x,N_V(N_V(u)))-(N_V \circ l_V)(x,N_V(u))\bigg )\\
&=N_V\bigg ( l_V(N(Nx),u)-(N_V \circ l_V)(Nx,u)+l_V(Nx,N_V(u))\\&+l_V(Nx,N_V(u))-(N_V \circ l_V)(x,N_V(u))+l_V(x,N_V(N_V(u)))\\
&-(N_V \circ l_V)(Nx,u)+(N_V \circ (N_V \circ l_V))(x,u)-(N_V \circ l_V)(x,N_V(u))\bigg)\\
&=N_V\bigg( l_V(N(Nx),u)-(N_V \circ l_V)(Nx,u)+l_V(Nx,N_V(u))\bigg)\\
&+N_V \bigg (  l_V(Nx, N_V(u))-(N_V \circ l_V)(x,N_V(u))+l_V(x,N_V(N_V(u)))\bigg )\\
&-(N_V \circ N_V)\bigg(l_V (Nx,u)-(N_V \circ l_V)(x,u)+l_V(x,N_V(u))\bigg )\\
&=(N_V \circ l_V^{'})(Nx,u)+(N_V \circ l_V^{'})(x,N_V(u))-(N_V \circ (N_V \circ l_V^{'}))(x,u).\\
\end{align*}
Hence, 
\[ l_V^{'}(N(x),N_V(u))=N_V(l_V^{'}(N(x),u)+l_V^{'}(x,N_V(u)))\]
holds for all $a \in \mathfrak{g}$ and $u \in V$. Similarly, for all $x \in \mathfrak{g}$ and $u \in V$, the following equation holds:
\[r_V^{'}(N_V(u),N(x))=N_V(r_V^{'}(N_V(u),x)+r_V^{'}(u,N(x))).\]

Hence, $(V,l_V^{'},r_V^{'},N_V )$ is a representation of the Nijenhuis Leibniz algebra $(\mathfrak{g}_N,[~,~]_*).$

\end{proof}

\section{Cohomology of Nijenhuis Leibniz algebras}\label{chap4}
Consider a Leibniz algebra $\mathfrak{g}$ with a representation $(V,l_V,r_V).$  Now, for any $n \geq 0$, we define an abelian group $C^n_{LA}(\mathfrak{g},V):=Hom(\mathfrak{g}^{\otimes n},V)$ and a map $\delta ^n: C^n_{LA}(\mathfrak{g},V) \rightarrow C^{n+1}_{LA}(\mathfrak{g},V)$  defined by
\begin{align*}
 &(\delta^n(f))(x_1,x_2,\ldots ,x_{n+1})\\
 & =\sum_{i=1}^{n}(-1)^{i+1}l_{V}(x_i,f(x_1,\ldots,\hat{x_i},\ldots, x_{n+1}))+(-1)^{n+1}r_{V}(f(x_1,\ldots,x_n),x_{n+1}) \\
 &+ \sum_{1 \leq i <j \leq n+1}(-1)^if(x_1,\ldots, \hat{x_i},\ldots ,x_{j-1},[x_i,x_j],x_{j+1},\ldots, x_{n+1}),
 \end{align*}
where $f\in C^n_{LA}(\mathfrak{g},V)$ and $x_1,\ldots,x_{n+1} \in \mathfrak{g}$. 

It is well known that $\{C_{LA}^{n}(\mathfrak{g},V),\delta^n \}$ is a cochain complex (For details see Loday-Pirashvili cohomology  for Leibniz algebra in \cite{Loday}). The corresponding cohomology groups are called the cohomology of $\mathfrak{g}$ with coefficients in the representation $V$ and the $n$th cohomology group is denoted by $H^n_{LA}(\mathfrak{g},V).$
We will  follow the notation $l_{V}(x,u)=[x,u] $ and $r_{V}(u,x)=[u,x]$ for all $x\in \mathfrak{g},~u \in V$  . Then the above coboundary map  $\delta ^n: C^n_{LA}(\mathfrak{g},V) \rightarrow C^{n+1}_{LA}(\mathfrak{g},V)$ becomes
 \begin{align*}
 (\delta^n(f))(x_1,x_2,\ldots ,x_{n+1})\\
 & =\sum_{i=1}^{n}(-1)^{i+1}[x_i,f(x_1,\ldots,\hat{x_i},\ldots, x_{n+1})]+(-1)^{n+1}[f(x_1,\ldots,x_n),x_{n+1}] \\
 &+ \sum_{1 \leq i <j \leq n+1}(-1)^if(x_1,\ldots, \hat{x_i},\ldots ,x_{j-1},[x_i,x_j],x_{j+1},\ldots, x_{n+1}),
  \end{align*}
where $f\in C^n_{LA}(\mathfrak{g},V)$ and $x_1,\ldots,x_{n+1} \in \mathfrak{g}$. \\
Let $\mathfrak{g}_N$ be a Nijenhuis Leibniz algebra with representation $(V,l_V,r_V,N_V)$. Now by Proposition \ref{prop basic} and \ref{prop coh}, we get a new Nijenhuis Leibniz algebra $(\mathfrak{g}_N,[~,~]_*)$ with representation $(V,l^{'}_V,r_V^{'},N_V)$ induced by the Nijenhuis operator $N$. We consider the Loday-Pirashvili cochain complex of this induced Leibniz algebra $(\mathfrak{g},[~,~]_*)$ with representation $(V,l^{'}_V,r_V^{'})$
as follows:

For any  $n \geq 0$, define cochain groups $C^n_{NO}(\mathfrak{g},V):=Hom(\mathfrak{g}^{\otimes n},V)$ and  boundary map $\partial ^n: C^n_{NO}(\mathfrak{g},V) \rightarrow C^{n+1}_{NO}(\mathfrak{g},V)$  by
\begin{align*}
 &(\partial^n(f))(x_1,x_2,\ldots ,x_{n+1})\\
 &=\sum_{i=1}^{n}(-1)^{i+1}l^{'}_{V}(x_i,f(x_1,\ldots,\hat{x_i},\ldots, x_{n+1}))+(-1)^{n+1}r^{'}_{V}(f(x_1,\ldots,x_n),x_{n+1})\\
& + \sum_{1 \leq i <j \leq n+1}(-1)^if(x_1,\ldots, \hat{x_i},\ldots ,x_{j-1},[x_i,x_j]_{*},x_{j+1},\ldots, x_{n+1}),\\
&=\sum _{i=1}^{n}(-1)^{i+1}[Nx_i,f(x_1,\ldots,\hat{x_i},\ldots, x_{n+1})] -\sum _{i=1}^{n}(-1)^{i+1}N_V([x_i,f(x_1,\ldots,\hat{x_i},\ldots, x_{n+1})])\\
&+\sum _{i=1}^{n}(-1)^{i+1}[x_i,N_V(f(x_1,\ldots,\hat{x_i},\ldots, x_{n+1}))]\\
&+(-1)^{n+1}[f (x_1,\ldots ,x_{n}),N(x_{n+1})] -(-1)^{n+1}N_V([f (x_1,\ldots ,x_{n}),x_{n+1}])\\
&+(-1)^{n+1}[N_V(f (x_1,\ldots ,x_{n})),x_{n+1}]\\
&+\sum_{1\leq i< j\leq n+1}(-1)^if (x_1,\ldots,\hat{x_i},\ldots,x_{j-1},[Nx_i,x_j]_{\mathfrak{g}}+[x_i,Nx_j]_{\mathfrak{g}}-N([x_i,x_j]_{\mathfrak{g}}),x_{j+1},\ldots ,x_{n+1})\\
 \end{align*}
where $f\in C^n_{NO}(\mathfrak{g},V)$ and $x_1,\ldots,x_{n+1} \in \mathfrak{g}$.\\
The above map satisfies the condition $\partial^{n+1} \circ \partial^n=0$ as it a coboundary map for the induced Leibniz algebra $(\mathfrak{g},[~,~]_*).$
 Therefore, $\{C^{n}_{NO}(\mathfrak{g},V),\partial^n\}$ is a cochain complex. This cochain complex is called the cochain complex of the Nijenhuis operator $N$ and  the corresponding cohomology groups are called the cohomology of Nijenhuis operator $N$ with coefficients in the representation $V$ and is denoted by $H^n_{NO}(\mathfrak{g},V).$

\begin{definition}
Let $(\mathfrak{g}_N, [~,~])$ be a nijenhuis Leibniz algebra with a representation $(V,l_V,r_V,N_V)$. We define a map for $n\geq 1,$
 $\phi^n : C^n_{LA}(\mathfrak{g},V) \rightarrow C^n_{NO}(\mathfrak{g},V)$ 
 \begin{align*}
\mbox{by}~~\phi^n(f)(x_1,x_2,\ldots,x_n)&=f(Nx_1,Nx_2,\ldots, Nx_n)- (N_V\circ f)(x_1,Nx_2,\ldots,Nx_n)\\
&-(N_V\circ f)(Nx_1,x_2,Nx_3,\ldots,Nx_n) -\ldots -(N_V\circ f)(Nx_1,Nx_2,\ldots,Nx_{n-1},x_n)\\
&+N_V^2(f(x_1,x_2,\ldots,x_n)).
\end{align*}
and $\phi_0=Id_V.$
\end{definition}
\begin{lemma}
 $\phi^{n+1}(\delta^n(f))(x_1,x_2,x_3,\ldots,x_{n+1})=\partial^n(\phi^n(f))(x_1,x_2,x_3,\ldots,x_{n+1})$, where $f\in C^n_{LA}(\mathfrak{g},V)$ and $x_1,\ldots,x_{n+1} \in \mathfrak{g}$. 
\end{lemma}
\begin{proof}
The proof is exactly similar to Lemma 4.2 of \cite{MS}.
\end{proof}
From the above lemma we have the following commutative diagram

\[\begin{tikzcd}
	{C^1_{LA}(\mathfrak{g},V)} & {C^2_{LA}(\mathfrak{g},V)} & {C^n_{LA}(\mathfrak{g},V)} & {C^{n+1}_{LA}(\mathfrak{g},V)} & {} \\
	{C^1_{NO}(\mathfrak{g},V)} & {C^2_{NO}(\mathfrak{g},V)} & {C^n_{NO}(\mathfrak{g},V)} & {C^{n+1}_{NO}(\mathfrak{g},V)} & {}
	\arrow[from=1-1, to=1-2,]{r}{\delta^1}
	\arrow[from=1-1, to=2-1]{d}{\phi^1}
	\arrow[from=1-2, to=2-2]{d}{\phi^2}
	\arrow[from=2-1, to=2-2]{r}{\partial^1}
	\arrow[dotted, no head, from=1-2, to=1-3]
	\arrow[dotted, no head, from=2-2, to=2-3]
	\arrow[from=1-3, to=1-4]{r}{\delta^n}
	\arrow[from=1-3, to=2-3]{d}{\phi^n}
	\arrow[from=2-3, to=2-4]{d}{\partial^n}
	\arrow[from=1-4, to=2-4]{r}{\phi^{n+1}}
	\arrow[dotted, no head, from=1-4, to=1-5]
	\arrow[dotted, no head, from=2-4, to=2-5]
\end{tikzcd}\]

\begin{definition}
Let $\mathfrak{g}_N$ be a Nijenhuis Leibniz algebra with a representation  $(V,l_V,r_V,N_V)$. We define 
\[C^{0}_{NLA}(\mathfrak{g},V):= C^{0}_{LA}(\mathfrak{g},V) ~~ \mbox{and}~~ C^n_{NLA}(\mathfrak{g},V):= C^n_{LA}(\mathfrak{g},V)\oplus C_{NO}^{n-1}(\mathfrak{g},V), \forall n \geq 1,\] and we also define a map
$d^n: C^n_{NLA}(\mathfrak{g},V) \rightarrow C^{n+1}_{NLA}(\mathfrak{g},V)$  by
\[d^n (f , g)=(\delta^n(f),- \partial^{n-1}(g)-\phi^n(f) )\]
for any $f \in C^n_{LA}(\mathfrak{g},V)$ , $g \in C^{n-1}_{NO}(\mathfrak{g},V).$
\end{definition}

\begin{theorem}
The map $d^n : C^n_{NLA} \rightarrow C^{n+1}_{NLA}$ satisfies $d^n \circ d^{n+1}=0.$
\end{theorem}
\begin{proof}
 Let, $f\in C^n_{LA}(\mathfrak{g},V) $ and $g\in C^{n-1}_{NO}(\mathfrak{g},V)$
\begin{align*}
&d^{n+1} \circ d^n(f,g)=d^{n+1}\bigg(\delta^n(f),-\partial^{n-1}(g)- \phi^n(f)\bigg)\\
&=\bigg ( \delta ^{n+1}(\delta ^n(f)), -\partial ^n  \bigg (\delta^n(f),-\partial^{n-1}(g)- \phi^n(f) \bigg )-\phi^{n+1}(\delta ^n(f)) \bigg )\\
&=\Bigg(0, \bigg (\partial^{n} (\phi^n(f))-\phi^{n+1}(\delta^n(f))\bigg)\Bigg ) ~~(~\mbox{by Lemma 3.2 })\\
&=\bigg(0,0\bigg).
\end{align*}
\end{proof}

Therefore, $\{C^n_{NLA}(\mathfrak{g},V),d^n\}$ forms a cochain complex. This cochain complex is called the cochain complex of the Nijenhuis Leibniz algebra $\mathfrak{g}_N$ with representation $(V,l_V,r_V,N_V)$. Suppose $Z^n_{NLA}$ denotes the space of $n$-th cocycle and $B^n_{NLA}$ denotes the space of $n$-th coboundaries. We define the quotient group  by 
\[H_{NLA}^n(\mathfrak{g},V):=\frac{Z_{NLA}^n(\mathfrak{g},V)}{B_{NLA}^n(\mathfrak{g},V)}
,~~ \mbox{for} ~n \geq 0.\]

We call the above quotient groups are the cohomology of Nijenhuis Leibniz algebra $\mathfrak{g}_N$ with representation $(V,l_V,r_V,N_V)$.

 \section{Deformation of Nijenhuis Leibniz algebra}\label{chap5}
 In this section, we study formal one-parameter deformation of Nijenhuis Leibniz algebra. We denote the bracket $[~,~]$ by $\mu.$

\begin{definition}
A one-parameter formal deformation of a Nijenhuis algebra $(\mathfrak{g}_N,\mu)$ is a pair of two power series  $(\mu_t,N_t)$
\[ \mu_t=\sum _{i=0}^{\infty}\mu_it^i, ~ \mu_{i} \in C^2_{LA}(\mathfrak{g},\mathfrak{g}),~~~~ N_t= \sum_{i=0}^{\infty} N_it^i,~ N_i \in C^1_{NO}(\mathfrak{g},\mathfrak{g}),\]
 such that $(\mathfrak{g}[[t]]_{N_t},\mu_t)$ is a Nijenhuis Leibniz algebra with $(\mu_0,N_{0})=(\mu ,N)$, where $\mathfrak{g}[[t]]$, the space of formal power series in $t$ with coefficients from $\mathfrak{g}$ is a $\mathbb{K}[[t]]$ module, $\mathbb{K}$ being the ground field of $(\mathfrak{g}_N,\mu)$.
\end{definition}
The above definition holds if and only if for any $x,y,z \in \mathfrak{g}$ the following conditions are satisfied 
\[\mu_t(x,\mu_t(y,z))=\mu_t(\mu_t(x,y),z)+\mu_t(y,\mu_t(x,z)),\]
and \[ \mu_t(N_t(x),N_t(y))=N_t\bigg(\mu_t(x,N_t(y))+\mu_t(N_t(x),y)-N_t(\mu_t (x,y))\bigg).\]

Expanding the above equations and equating the coefficients of $t^n$ from both sides we have 
\begin{align}\label{deform 1}
\sum _{\substack{i+j=n \\i,j\geq 0}} \mu_i(x,\mu_j (y,z))=\sum _{\substack{i+j=n \\i,j\geq 0}}\mu_i(\mu _j(x,y),z)+\sum _{\substack{i+j=n \\i,j\geq 0}} \mu_i(y, \mu_j (x,z)),
\end{align}
and 
\begin{align}\label{deform 2}
\sum_{\substack{i+j+k=n \\ i,j,k \geq 0}} \mu_i(N_j(u),N_k(v))&=\sum_{\substack{i+j+k=n \\ i,j,k \geq 0}}N_i(\mu_j(N_k(u),v))+\sum_{\substack{i+j+k=n \\ i,j,k \geq 0}}N_i(\mu_j(u,N_k(v)))\\
&-\sum_{\substack{i+j+k=n \\ i,j,k \geq 0}}N_iN_j(\mu_k(u,v)) \nonumber
\end{align}
Observe that, for $n=0$, the above conditions are exactly the conditions in the definitions of  Leibniz algebra and the Nijenhuis operator.
\begin{proposition}
Let $(\mu_t,N_t)$ be a one-parameter deformation of a Nijenhuis Leibniz algebra $(\mathfrak{g}_N,\mu)$. Then $(\mu_1,N_1)$ is a $2$- cocyle in the cochain complex $\{C^n_{NLA}(\mathfrak{g}),d^n\}.$
\end{proposition}

\begin{proof}
First we put $n=1$ in equation \ref{deform 1}, we get 
\[ \mu (x,\mu_1(y,z))+\mu_1(x,\mu (y,z))= \mu (\mu_1(x,y),z)+\mu_1(\mu (x,y),z) \\
+\mu_1 (y, \mu (x,z))+\mu (y, \mu_1(x,z))\]
Hence, we have $\delta^2(\mu_1)(x,y,z)=0 \in C^2_{LA}(\mathfrak{g})$.\\
Now, putting $n=1$ in equation \ref{deform 2}, we get
\begin{align*}
&\mu_1(Nu,Nv)+\mu (N_1u,Nv)+\mu(Nu,N_1v)\\
&=N_1(\mu(Nu,v))+N(\mu_1(Nu,v))+N(\mu (N_1u,v))\\
& N_1(\mu (u,Nv))+N(\mu_1(u,Nv))+N(\mu (u,N_1v))\\
&-N_1N(\mu (u,v))-NN_1(\mu (u,v))-N^2(\mu (u,v)).
\end{align*}
Hence,
\begin{align*}
&-N_1N(\mu (u,v))-\mu (N_1u,Nv)-\mu (Nu,N_1v)+N_1(\mu (u,Nv))\\
&N_1(\mu (Nu,v))+N(\mu (N_1u,v))+N(\mu (u,N_1v))-\mu (u,NN_1v)-\mu (NN_1u,v)\\
&=\mu_1(Nu,Nv)-N(\mu_1(Nu,v))-N(\mu_1(u,Nv))+N^2(\mu_1(u,v))\\
&+NN_1 (\mu (u,v))-\mu (u,NN_1(v))-\mu (NN_1u,v).\\
\end{align*}
Therefore,
\[-\partial^1(N_1)(u,v)= \phi^2(\mu_1)(u,v).\]
or, \[-\partial^1(N_1)(u,v)- \phi^2(\mu_1)(u,v)=0.\]
Hence, $d^2(\mu_1,N_1)=0$. Therefore, $(\mu_1,N_1)$ is a $2$-cocycle in the cochain complex $C^n_{NLA}(\mathfrak{g}, \mathfrak{g}).$

\end{proof}

\begin{definition}
Let $(\mu_t, N_t)$ and $(\mu_t^{'},N_t^{'})$ be two deformation of a Nijenhuis Leibniz algebra $(\mathfrak{g}_N, \mu)$. A formal isomorphism from $(\mu_t, N_t)$ to $(\mu_t^{'},N_t^{'})$ is a power series $\psi_t = \sum _{i=0}\psi_it^i : \mathfrak{g}[[t]] \rightarrow \mathfrak{g}[[t]]$, where $\psi_i: \mathfrak{g} \rightarrow \mathfrak{g}$ are linear maps with $\psi_0= Id_{\mathfrak{g}}$ satisfying the following conditions
\begin{align*}
&\psi_t \circ \mu_t^{'}=\mu_t \circ (\psi_t \circ \psi_t),\\
&\psi_t \circ N_t^{'}=N_t \circ \psi_t.
\end{align*}
\end{definition}
Here we call $(\mu_t, N_t)$ and $(\mu_t^{'},N_t^{'})$ are equivalent. Now from the above conditions, we have
\begin{align}
\label{deform 3}&\sum_{\substack {i+j=n \\ i,j\geq 0}}\psi _i(\mu_j^{'}(x,y))=\sum_{\substack {i+j+k=n \\ i,j,k\geq 0}}\mu_i(\psi_j(x),\psi_k(y)),~~ x,y \in \mathfrak{g}.\\
\label{deform 4}& \sum_{\substack {i+j=n \\ i,j\geq 0}}\psi _i \circ N^{'}_j=\sum_{\substack {i+j=n \\ i,j\geq 0}} N_i \circ \psi _j.
\end{align}

\begin{theorem}\label{thm 4}
The infinitesimal of two equivalent formal one-parameter deformation of Nijenhuis Leibniz algebra $(\mathfrak{g}_N , \mu )$ are in the same cohomology class. 
\end{theorem}
\begin{proof}
Let $(\mu_t, N_t)$ and $(\mu_t^{'},N_t^{'})$ are two equivalent formal deformation of a Nijenhuis Leibniz algebra $(\mathfrak{g}_N , \mu )$ with a formal isomorphism $\psi_t : (\mu_t,N_t) \rightarrow  (\mu_t^{'},N_t^{'})$. Now puting $n=1$ in equation \ref{deform 3} and \ref{deform 4}, we get
\begin{align*}
&\mu^{'}_1(x,y)=\mu_1(x,y)+\mu (x,\psi_1 (y))+\mu (\psi_1(x),y)-\psi_1(\mu (x,y)) ,~~ x,y \in \mathfrak{g}\\
& N_1^{'}=N_1+N \circ \psi_1-\psi _1 \circ N.
\end{align*}
Hence,
\[(\mu_1^{'},N_1^{'})-(\mu_1,N_1)=(\delta^1(\psi_1),-\phi^1(\psi_1))=d^1(\psi_1,0) \in C^{1}_{NLA}(\mathfrak{g},\mathfrak{g}).\]
\end{proof}

\begin{theorem}
Let $(\mathfrak{g}_N,\mu)$ be a Nijenhuis Leibniz algebra. If $H^2_{NLA}(\mathfrak{g},\mathfrak{g})=0$, then $(\mathfrak{g}_N,\mu)$ is rigid.
\end{theorem}
\begin{proof}
Let $(\mu_t,N_t)$ be a deformation of $(\mathfrak{g}_N,\mu)$. Since, $(\mu_1,N_1)$ is a $2$-cocyle and $H^2_{NLA}(\mathfrak{g},\mathfrak{g})=0$, thus, there exists a map $\psi_1^{'}$ and $x\in \mathbb{K} $, where $\mathbb{K}$ is the ground field of  Nijenhuis Leibniz algebra $(\mathfrak{g}_T,\mu)$, such that
\[(\psi_1^{'},x) \in C^1_{NLA}(\mathfrak{g},\mathfrak{g})=C^1_{LA}(\mathfrak{g},\mathfrak{g})\oplus Hom(\mathbb{K},\mathfrak{g}),\]
and $(\mu_1,N_1)=d^1(\psi_1^{'},x)$. Hence, $\mu_1=\delta^1(\psi_1^{'})~ \mbox{and}~N_1=-\partial^0(x)-\phi^1(\psi_1)$. Let $\psi_1=\psi_1^{'}+\delta^0(x)$, then $\mu_1=\delta^1(\psi_1),~ N_1=-\phi^1(\psi_1)$. Let $\psi_t=Id_{\mathfrak{g}}-t\psi_t$. Then we have two equivalent deformation 
$(\mu_t,N_t)$ and $(\bar{\mu_t},\bar{N_t})$, where
\[\bar{\mu_t}=\psi_t^{-1} \circ \mu_t \circ (\psi_t \times \psi_t), ~~ \bar{N_t}=\psi_t^{-1} \circ N_t \circ \psi_t.
\]
By Theorem \ref{thm 4}, we have $\bar{\mu_1}=0,\bar{N_1}=0$. Hence,
\begin{align*}
&\bar{\mu_t}=\mu +\bar{\mu_2}t^2+\ldots ,\\
& \bar{N_t}=N+\bar{N_2}t^2+\ldots ~.
\end{align*}
Thus, the linear terms of  $(\bar{\mu_2},\bar{N_2})$ vanishes, hence, repeatedly applying the same argument we conclude that $(\mu_t,N_t)$ is equivalent to the trivial  deformation. Hence, $(\mathfrak{g}_N,\mu)$ is rigid. 

\end{proof}

\section{Abelian Extension of Nijenhuis Leibniz algebra}\label{chap6}
Let $V$ be a vector space and $\mathfrak{g}_N$ be a Nijenhuis Leibniz algebra. Then for any linear operator $N_V$ on $V$ with the trivial bracket $\mu$ defined by $\mu : V \times V \rightarrow V$ by $\mu(x,y)=0$ for all $x,y \in V$ makes $(V_{N_V},\mu)$ a Nijenhuis Leibniz algebra.

\begin{definition}
Let $(\mathfrak{g}_N, [~,~])$ be a Nijenhuis Leibniz algebra and $V$ be any vector space with a linear operator $N_V$ with trivial bracket $\mu$. Then the abelian extension of $(\mathfrak{g}_N, [~,~])$ is a short exact sequence of morphisms of Nijenhuis Leibniz algebra

 \[
\begin{tikzcd}
0 \arrow[r] & (V_{N_V},\mu) \arrow[r ,"i"] & (\hat{\mathfrak{g}}_{\hat{N}},[~,~]_{\wedge}) \arrow[r,"p"] & (\mathfrak{g}_{N},[~,~]) \arrow [r] & 0 .
\end{tikzcd} 
\]

 In this case we say that  $(\hat{\mathfrak{g}}_{\hat{N}},[~,~]_{\wedge})$ is an abelian extension of $(\mathfrak{g}_N, [~,~])$  by $(V_{N_V}, \mu )$.
\end{definition}
From the above definition it follows that for an abelian extension of $(\mathfrak{g}_N, [~,~])$  there exists a commutative diagram 

\[
\begin{tikzcd}
0 \arrow[r] & V \arrow[r ,"i"] \arrow[d,"N_V"]& \hat{\mathfrak{g}} \arrow[d,"\hat{N}"]\arrow[r,"p"] & \mathfrak{g} \arrow [r]\arrow[d,"N"]  & 0 \\
 0 \arrow[r] & V \arrow[r ,"i"] & \hat{\mathfrak{g}} \arrow [r,"p"]  &  \mathfrak{g} \arrow [r]  & 0
\end{tikzcd}
\]
where $\mu (a,b)=0$ for all $a,b \in V.$

\begin{definition}
Let $(\hat{\mathfrak{g}}_{\hat{N_1}},[~,~]_{\wedge_1})$ and $(\hat{\mathfrak{g}}_{\hat{N_2}},[~,~]_{\wedge_2})$ be two abelian extension of $(\mathfrak{g}_{N},[~,~])$ by $(V_{N_V},\mu)$. Then this two extension are said to be isomorphic if there exists an isomorphism of Nijenhuis Leibniz algebra $\xi : (\hat{\mathfrak{g}}_{\hat{N_1}},[~,~]_{\wedge_1}) \rightarrow  (\hat{\mathfrak{g}}_{\hat{N_2}},[~,~]_{\wedge_2}) $ so that the following diagram is commutative : 

\[
\begin{tikzcd}
0 \arrow[r] & (V_{N_V},\mu) \arrow[r ,"i"] \arrow[d,equal]& (\hat{\mathfrak{g}}_{\hat{N_1}},[~,~]_{\wedge_1})\arrow[d,"\xi"]\arrow[r,"p"] & (\mathfrak{g}_N,[~,~]) \arrow [r]\arrow[d,equal]  & 0 \\
 0 \arrow[r] & (V_{N_V},\mu) \arrow[r ,"i"] & (\hat{\mathfrak{g}}_{\hat{N_2}},[~,~]_{\wedge_2}) \arrow [r,"p"]  &  (\mathfrak{g}_N,[~,~]) \arrow [r]  & 0.
\end{tikzcd}
\]

\end{definition}
\begin{definition}
Let $(\hat{\mathfrak{g}}_{\hat{N}},[~,~]_{\wedge})$ be an abelian extension of a Nijenhuis Leibniz algebra $(\mathfrak{g}_N,[~,~])$ by $(V_{N_V}, \mu)$. A linear map $s: \mathfrak{g} \rightarrow \hat{\mathfrak{g}}$ is called a section of this abelian extension if $p \circ s =Id_{\mathfrak{g}}.$
\end{definition}
\begin{theorem}
Let $(\hat{\mathfrak{g}}_{\hat{N}},[~,~]_{\wedge})$ be an abelian extension of $(\mathfrak{g}_N,[~,~])$ by $(V_{N_V},\mu)$ with a section $s: \mathfrak{g} \rightarrow \hat{\mathfrak{g}}$. Now suppose $\bar{l}_V: \mathfrak{g} \otimes V \rightarrow V$ and $\bar{r}_V: V \otimes \mathfrak{g} \rightarrow V $ be two maps defined by   $\bar{l}_V(x,u)=[s(x),u]_{\wedge}$ and $\bar{r}_V(u,x)=[u,s(x)]_{\wedge}$ for all $x\in \mathfrak{g}, u \in V.$ Then,  $(V,\bar{l}_V,\bar{r}_V,N_V)$ is a representation of Nijenhuis Leibniz algebra  $(\mathfrak{g}_N,[~,~]).$
\end{theorem}
\begin{proof}
As $s([x,y])-[s(x),s(y)]_{\wedge} \in V$, hence $[s([x,y]),u]_{\wedge}=[[s(x),s(y)]_{\wedge},u]_{\wedge}$ for all $x,y \in \mathfrak{g}, u \in V$. Therefore, we have
\begin{align*}
&\bar{l}_V(x,\bar{l}_V(y,u))-\bar{l}_V([x,y],u)-\bar{l}_V(y,\bar{l}_V(x,u)) \\
&= \bar{l}_V(x,[s(y),u]_{\wedge})-[s([x,y]),u]_{\wedge}-\bar{l}_V(y,[s(x),u]_{\wedge}).\\
&=[s(x),[s(y),u]_{\wedge}]_{\wedge}-[s([x,y]),u]_{\wedge}-[s(y),[s(x),u]_{\wedge}]_{\wedge}.\\
&=[s(x),[s(y),u]_{\wedge}]_{\wedge}-[[s(x),s(y)]_{\wedge},u]_{\wedge}-[s(y),[s(x),u]_{\wedge}]_{\wedge}.\\
&=0.
\end{align*}
Similarly, we can show that 
\[\bar{l}_V(x,\bar{r}_V(u,y))=\bar{r}_V(\bar{l}_V(x,u),y)+\bar{r}_V(u,[x,y])\]
\[ \bar{r}_V(u,[x,y]= \bar{r}_V(\bar{r}_V(u,x),y)+\bar{l}_V(x,\bar{r}_V(u,y))\]
for all $x,y \in \mathfrak{g}$ and $u\in V.$ Hence, $(V,\bar{l}_V, \bar{r}_{V})$ is a representation of a Leibniz algebra $(\mathfrak{g}, [~,~]).$
Now, $s(N(x))-\hat{N}(s(x)) \in V$, hence, $[s(N(x)),u]=[\hat{N}(s(x)),u]$ for all $x\in \mathfrak{g}, u \in V$. Therefore, we have \begin{align*}
&\bar{l}_V(N(x),N_V(u))=[s(N(x)), N_V(u)]_{\wedge}=[\hat{N}(s(x)),\hat{N}(u)]_{\wedge}\\
&=\hat{N}\bigg([\hat{N}(s(x)),u]_{\wedge}+[s(x),\hat{N}(u)]_{\wedge}-\hat{N}([s(x),u])\bigg)\\
&=N_V\bigg([s(N(x)),u]_{\wedge}+[s(x),N_V(u)]_{\wedge}-N_V([s(x),u]_{\wedge})\bigg)\\
&=N_V\bigg(\bar{l}_V(Nx,u)+\bar{l}_V(x,N_V(u))-(N_V \circ \bar{l}_V)(x,u)\bigg).
\end{align*}

Hence, \[\bar{l}_V(Nx,N_V(u))=N_V\bigg(\bar{l}_V(Nx,u)+\bar{l}_V(x,N_V(u))-(N_V \circ \bar{l}_V)(x,u)\bigg)\] for all $x,y \in \mathfrak{g}$ and $u\in V.$

Similarly, it can be shown that 
\[\bar{r}_V(N_V(u),Nx)=N_V\bigg(\bar{r}_V(N_V(u),x)+\bar{r}_V(u,Nx)-(N_V \circ \bar{r}_V)(u,x)\bigg)\]
for all $x \in \mathfrak{g}$ and $u \in V.$ Hence, $(V,\bar{l}_V,\bar{r}_V,N_V)$ is a representation of Nijenhuis Leibniz algebra  $(\mathfrak{g}_N,[~,~]).$

\end{proof}
\begin{proposition}
Let $(\hat{\mathfrak{g}}_{\hat{N}},[~,~]_{\wedge})$ be an abelian extension of $(\mathfrak{g}_N,[~,~])$ by $(V_{N_V},\mu)$. Then any two distinct sections $s_1,s_2: \mathfrak{g} \rightarrow \hat{\mathfrak{g}}$ give the same Nijenhuis Leibniz algebra representation $(V,\bar{l}_V,\bar{r}_V,N_V).$ 
\end{proposition}
\begin{proof}
Let $s_1$ and $s_2$ be any two distinct  sections. Now define $\gamma : \mathfrak{g} \rightarrow V$ by
\[\gamma(x)=s_1(x)-s_2(x),~\mbox{for all }~ x \in \mathfrak{g}.\]
Since, $\mu (u,v)=0$ for all $ u,v \in V.$
Thus, 
\begin{align*}
[s_1(x),u]_{\wedge}=[\gamma(x)+s_2(x),u]_{\wedge}=[\gamma(x),u]_{\wedge}+[s_2(x),u]_{\wedge}=[s_2(x),u]_{\wedge}.
\end{align*}
Similarly, $[u,s_1(x)]_{\wedge}=[u,s_2(x)]_{\wedge}$ for all $x,y \in \mathfrak{g}, u \in V.$
Thus, two distinct sections gives the same representation.
\end{proof}

\begin{proposition}
Let $(\hat{\mathfrak{g}}_{\hat{N}},[~,~]_{\wedge})$ be an abelian extension of $(\mathfrak{g}_N,[~,~])$ by $(V_{N_V},\mu)$. Let  $\psi : \mathfrak{g}\otimes \mathfrak{g} \rightarrow V$ and $\chi : \mathfrak{g} \rightarrow V$ be two maps defined by
 \begin{align*}
&\psi (x \otimes y)=[s(x),s(y)]_{\wedge}-s([x,y]),\\
&\chi (x)=\hat{T}(s(x))-s(T(x)), ~~ \mbox{for all}~ x,y \in \mathfrak{g}, \mbox{respcetively.}
\end{align*}
Then, the cohomological class of $(\psi,\chi)$ is independent of the choice of section.
\end{proposition}
\begin{proof}
Let $s_1$ and $s_2$ be two distinct  sections. Define $\gamma : \mathfrak{g} \rightarrow V$ by
\[\gamma(x)=s_1(x)-s_2(x),~\mbox{for all }~ x \in \mathfrak{g}.\]
Since, $\mu (u,v)=0$ for all $ u,v \in V.$
Therefore, for all $x,y \in \mathfrak{g}, u \in V,$
\begin{align*}
\psi_1(x,y)
&=[s_1(x),s_1(y)]_{\wedge}-s_1([x,y])\\
&=[s_2(x)+\gamma(x),s_2(y)+\gamma(y)]_{\wedge}-(s_2([x,y])+\gamma ([x,y]_{\mathfrak{g}}))\\
&=[s_2(x),s_2(y)]_{\wedge}-s_2([x,y])+[\gamma(x),s_2(y)]_{\wedge}+[s_2(x),\gamma (y)]_{\wedge}-\gamma([x,y])\\
&=\psi_2(x,y)+\delta^1 (\gamma)(x,y).
\end{align*}
Again,
\begin{align*}
\chi_1(x)=\hat{N}(s_1(x))-s_1(N(x))&=\hat{N}(s_2(x)+\gamma(x))-s_2(N(x))-\gamma(N(x))\\
&=\chi_2(x)+N_V(\gamma(x))-\gamma (N(x))=\chi_2(x)-\phi^1(\gamma)(x).
\end{align*}
Therefore, $(\psi_1,\chi_1)-(\psi_2,\chi_2)=(\delta^1 (\gamma),~-\phi^1(\gamma))=d^1(\gamma)$. Hence, $(\psi_1,\chi_1)$ and $(\psi_2,\chi_2)$ are in the same cohomology class $H^2_{NLA}(\mathfrak{g},V)$.
\end{proof}
\begin{proposition}
Two isomorphic abelian extension of a Nijenhuis Leibniz algebra $(\mathfrak{g}_N,[~,~])$ by $(V_{N_V},\mu)$ give rise to the same element in $H^2_{NLA}(\mathfrak{g},V)$.
\end{proposition}
\begin{proof}
Let $(\hat{\mathfrak{g}}_{\hat{N_1}},[~,~]_{\wedge_1})$ and $(\hat{\mathfrak{g}}_{\hat{N_2}},[~,~]_{\wedge_2})$ be two isomorphic abelian extension of $(\mathfrak{g}_N,[~,~])$ by $(V_{N_V},\mu)$. Let $s_1$ be a section of $(\hat{\mathfrak{g}}_{\hat{N_1}},[~,~]_{\wedge_1})$. Thus, we have 
$p_2 \circ (\xi \circ s_1)=p_1 \circ s_1 =Id_{\mathfrak{g}}$ as $p_2 \circ \xi =p_1$, where $\xi$ is the map between the two abelian extensions. Hence, $\xi \circ s_1$ is a section of $(\hat{\mathfrak{g}}_{\hat{N_2}},[~,~]_{\wedge_2})$.
Now define $s_2 :=\xi \circ s_1$. Since, $\xi$ is a homomorphism of Nijenhuis Leibniz algebras such that 
$\xi|_{V}=Id_V, ~\xi ([s_1(x),u]_{\wedge_1})=[s_2(x),u]_{\wedge_2}$ . Thus, $\xi|_{V}: V \rightarrow V$ is compatible with the induced representations.\\
Now, for all $x,y \in \mathfrak{g}$
\begin{align*}
&\psi_2(x \otimes y)=[s_2(x),s_2(y)]_{\wedge_2}-s_2([x,y])=[\xi(s_1(x)),\xi(s_1(y)])_{\wedge_2}-\xi(s_1([x,y]))\\
&=\xi ([s_1(x),s_1(y)]_{\wedge_1}-s_1([x,y]))=\xi (\psi_1 (x \otimes y))=\psi_1(x\otimes y)
\end{align*}
and
\begin{align*}
&\chi_2(x)= \hat{N_2}(s_2(x))-s_2(N(x))=\hat{N_2}(\xi(s_1(x)))-\xi(s_1(N(x)))\\
&=\xi (\hat{N_1 }(s_1(x))-s_1(N(x)))=\xi(\chi_1(x))=\chi_1(x).
\end{align*}
Therefore, two isomorphic abelian extension give rise to the same element in $H^2_{NLA}(\mathfrak{g},V).$
\end{proof}

{\bf Acknowledgements:} The second author is supported by the Core Research Grant (CRG) of Science and Engineering Research Board (SERB), Department of Science and Technology (DST), Govt. of India. (Grant Number- CRG/2022/005332)

\renewcommand{\refname}{REFERENCES}


\begin{thebibliography}{99}

\bibitem{CN04}
J. Clemente-Gallardo, J. M. Nunes Da Costa: Dirac-Nijenhuis Structures, \emph{J. Phy. A} 37 (2004), 7267--7296.

\bibitem {Das} A. Das, A cohomological study of modified Rota-Baxter algebras,  \emph{arXiv:2207.02273v1}.

\bibitem{DS}
A. Das, S. Sen. Nijenhuis operators on Hom-Lie algebras. \emph{Commun. Algebra} 50 (2022): 1038–1054.

\bibitem {Demir}
I. Demir, K. C. Mishra, E. Stitzinger, On some structures of Leibniz algebras, \emph{Contemporary Mathematics}, 623(2014), 41-54.

\bibitem{Dorfman1993}
I.  Dorfman, Dirac structures and integrability of nonlinear evolution equations, \emph{Wiley, Chichester}, 1993.

\bibitem{gers2}
M. Gerstenhaber, {The cohomology structure of an associative ring}, {\em Ann. of Math.} (2) 78 (1963) 267--288.

\bibitem{gers} M. Gerstenhaber, {On the deformation of rings and algebras}, {\em Ann. of Math.} (2) 79 (1964) 59--103.

\bibitem{Guo} S. Guo, Y. Qin, K. Wang, G. Zhou, Cohomology theory of Rota-Baxter pre-Lie algebras of arbitrary weights, arXiv:2204.13518v3.

\bibitem{HCM}
A. Ben Hassine, T. Chtioui and S. Mabrouk, Nijenhuis operators on 3-Hom-L-dendriform algebras, \emph{Commun. Algebra} 49(12) (2021) 5367–5391.

\bibitem{kos}
Y. Kosmann-Schwarzbach, Nijenhuis structures on Courant algebroids, \emph{Bull Braz Math Society}, New series 42(4),625-649, (2011).

\bibitem{LSZB}
J. Liu, Y. Sheng, Y. Zhou,  C. Bai. Nijenhuis Operators onn-Lie Algebras. \emph{Communications in Theoretical Physics}, 65(6)(2016), 659--670

\bibitem{Loday}J. L. Loday, T. Pirashvili, Universal enveloping algebras of Leibniz algebras and (co)homology, \emph{Math. Ann.} 296 (1993), 139-158.

\bibitem{LB04}
H. Longguang, L. Baokang,  Dirac Nijenhuis Manifolds, \emph{Reports in Mathematical Physics}, 53, 1, 123-142 (2004).

\bibitem{MS-mody}
B. Mondal, R. Saha, Cohomology of modified Rota-Baxter Leibniz algebra of weight $k$ (2022),  \emph{arXiv:2211.07944}.

\bibitem {MS} B. Mondal, R. Saha, Cohomology, deformations and extensions of Rota-Baxter Leibniz algebras, \emph{Communatications in Mathematics},  Volume 30 (2022), Issue 2.

\bibitem{Sun} Q. Sun, N. Jing, O-operators and related structure on Leibniz algebras, \emph{Communications in algebra}, 51 (2023), 2199-2216.

\bibitem {Tang} R. Tang, Y. Sheng, Y. Zhou, Deformations of relative Rota-Baxter operators on Leibniz algebras. \emph{Int. J. Geom. Methods Mod. Phys.} 17 (2020), 12-21.

\bibitem {Wang} Q. Wang, Y. Sheng, C. Bai, J. Liu, Nijenhuis operators on pre-Lie algebras, \emph{Communications in Contemporary Mathematics} Vol. 21, No. 07, 1850050 (2019).

\bibitem{zhang}
T. Zhang, Deformations and Extensions of 3-Lie Algebras, \emph{arXiv:1401.4656.}

\end{thebibliography}
\end{document}